\documentclass[12pt,reqno]{amsart}
\usepackage{amssymb,mathrsfs,graphicx}
\usepackage{ifthen}
\usepackage{mathtools}
\usepackage[margin=3cm]{geometry}
\mathtoolsset{showonlyrefs=true}

\title[Weakly singular Euler-Alignment system]{On the Euler-Alignment
    system with weakly singular communication weights}

\author[Changhui Tan]{Changhui Tan}
\address[Changhui Tan]{\newline Department of Mathematics, \ 
 University of South Carolina, 1523 Greene St., Columbia, SC 29208, USA}
\email{tan@math.sc.edu}



\newtheorem{theorem}{Theorem}[section]
\newtheorem{lemma}{Lemma}[section]

\newtheorem{proposition}{Proposition}[section]
\newtheorem{remark}{Remark}[section]

\newcommand{\R}{\mathbb R}

\newcommand{\pa}{\partial}

\def\grad{\nabla}
\def\div{\nabla\cdot}

\begin{document}
\allowdisplaybreaks

\begin{abstract}
We study the pressureless Euler equations with nonlocal alignment
interactions, which arises as a macroscopic representation of complex
biological systems modeling animal flocks. For such Euler-Alignment
system with bounded interactions, a critical threshold phenomenon is
proved in \cite{tadmor2014critical}, where global regularity depends
on initial data. With strongly singular interactions, global
regularity is obtained in \cite{do2018global}, for all initial data.
We consider the remaining case when the interaction is weakly
singular. We show a critical threshold, similar as the system with
bounded interaction. However, different global behaviors may happen for
critical initial data, which reveals the unique structure of the
weakly singular alignment operator.
\end{abstract}

\maketitle 
\tableofcontents

\section{Introduction}\label{sec1}
We are interested in the  Euler-Alignment system, which takes the form
\begin{align}
&\pa_t\rho+\div(\rho u)=0,\label{eq:EArho}\\
&\pa_tu+u\cdot\grad u=\int\psi(|x-y|)(u(y)-u(x))\rho(y)dy. \label{eq:EAu}
\end{align}

The system arises as a macroscopic representation of models
characterizing collective behaviors, in particular,
alignment and flocking.

Here, $\rho$ represents the density of the group, and $u$ is the
associated velocity. The term appears at the right hand side of
\eqref{eq:EAu} is the \emph{alignment force}. It was first proposed by
Cucker and Smale in \cite{cucker2007emergent} in the microscopic
model
\begin{equation}\label{eq:CS}
\dot{x}_i=v_i,\quad m\dot{v}_i=\frac{1}{N}\sum_{j=1}^N\psi(|x_i-x_j|)(v_j-v_i).
\end{equation}
$\psi: \R^+\to\R$ is called the \emph{communication weight}, measuring
the strength of the alignment interaction. A natural assumption on
$\psi$ is that it is a decreasing function, as the strength of
interaction is weaker when the distance is larger.

The alignment force in the Cucker-Smale system \eqref{eq:CS}  intends
to align the velocity of all particles as time becomes large. The
corresponding flocking phenomenon has been proved in
\cite{ha2009simple}, under appropriate assumptions on the
communication weight.

The Euler-Alignment system \eqref{eq:EArho}-\eqref{eq:EAu} can be
derived from the Cucker-Smale system \eqref{eq:CS}, through a kinetic
description, as a hydrodynamic limit. See \cite{ha2008particle} for a
formal derivation, \cite{carrillo2010asymptotic,tan2017discontinuous}
for discussions on the kinetic system,
and \cite{figalli2018rigorous,karper2015hydrodynamic} for
rigorous passages to the limit.

\subsection{Bounded interaction}
The Euler-Alignment system \eqref{eq:EArho}-\eqref{eq:EAu} with
bounded Lipschitz was first studied in \cite{tadmor2014critical}, where
a \emph{critical threshold phenomenon} is proved: subcritical initial
data lead to global smooth solutions, while supercritical initial data
lead to finite time singularity formations.

In a successive work \cite{carrillo2016critical}, a sharp critical
threshold condition is obtained in 1D, with the help of an important
quantity
\begin{equation}\label{def:G}
G(x,t) = \pa_xu(x,t)+\int\psi(x-y)\rho(y,t)dy.
\end{equation}

One can easily obtain the dynamics of $G$, see
\cite{carrillo2016critical}, as follows
\begin{equation}\label{eq:G1D}
\pa_tG+\pa_x(Gu)=0.
\end{equation}
This together with the dynamics of $\rho$
\begin{equation}\label{eq:rho1D}
\pa_t\rho+\pa_x(\rho u)=0,
\end{equation}
can serve as an alternative representation of
\eqref{eq:EArho}-\eqref{eq:EAu}. The velocity field $u$ can
be recovered by \eqref{def:G}.

The following theorem shows the sharp critical threshold condition.
\begin{theorem}[\cite{carrillo2016critical}]\label{thm:bounded}
Consider the 1D Euler-Alignment system \eqref{eq:G1D}-\eqref{eq:rho1D}
with smooth initial data $(\rho_0, G_0)$.
\begin{itemize}
\item If $\displaystyle\inf_x G_0(x)\geq0$, then there exists a globally regular
  solution.
\item If $\displaystyle\inf_x G_0(x)<0$, then the solution admits a finite time blowup.
\end{itemize}
\end{theorem}

For 2D Euler-Alignment system, the threshold conditions are obtained in
\cite{tadmor2014critical}, and also in \cite{he2017global} with a
further improvement. However, neither result is sharp.

\subsection{Strongly singular interaction}
One family of influence functions has the form
\begin{equation}
\psi(r) = r^{-s}.\label{eq:psi}
\end{equation}
When $s>0$, $\psi$ is unbounded at $r=0$. This corresponds to the case
when the alignment interaction becomes very strong as the distance
becomes smaller.

In the case when $s>n$, where $n$ is the dimension, $\psi(|x|)$ is not
integrable at $x=0$. It has been studied recently that the so called
strongly singular interaction has a regularization effect, which
prevents the solution from finite time singularity formations. In 1D, 
global regularity is obtained in \cite{do2018global} for $s\in(1,2)$,
and in \cite{shvydkoy2017eulerian} for $s\in[2,3)$ through a different approach.

\begin{theorem}[\cite{do2018global,shvydkoy2017eulerian}]\label{thm:strong}
Consider the 1D Euler-Alignment system \eqref{eq:G1D}-\eqref{eq:rho1D}
with smooth periodic initial data $(\rho_0, G_0)$. Suppose
$\rho_0>0$. Then, there exists a globally regular solution.
\end{theorem}

Note that since $\psi$ is not integrable, the quantity $G$ in
\eqref{def:G} is not well-defined. For $\psi$ defined in
\eqref{eq:psi}, one can use an alternative quantity
$G=\pa_xu-(-\Delta)^{(s-1)/2}\rho$.
For general choice of $\psi$ with the same singularity at $x=0$,
a similar global regularity result has been obtained in \cite{kiselev2018global}.

The dynamics in 2D is much more complicated and far less
understood. Global regularity has been obtained recently in
\cite{shvydkoy2018global} only for a small class of initial data.

\subsection{Weakly singular interaction}

We are interested in the Euler-Alignment system
\eqref{eq:EArho}-\eqref{eq:EAu} with weakly singular interactions.
This corresponds to the case when $\psi(|x|)$ is integrable, namely
$\psi(r)$ behaves like $r^{-s}$ near origin with $s\in(0,n)$.

In this case, the quantity $G$ is well-defined as long as the solution
$(\rho,u)$ is smooth, since
\[\|G\|_{L^\infty}\leq \|u\|_{W^{1,\infty}}+\|\psi\|_{L^1}\|\rho\|_{L^\infty}.\]
So in 1D, one would expect a similar critical threshold phenomenon as
Theorem \ref{thm:bounded}. However, the result is not always true.

Let us consider a special case when $G_0(x)\equiv0$. Since $G$
satisfies \eqref{eq:G1D}, it is easy to see that $G(x,t)=0$ in all
time. The dynamics of $\rho$ can be written as
\begin{equation}\label{eq:agg}
\pa_t\rho+\pa_x(\rho u)=0,\quad 
u(x,t) = -\int K'(x-y)\rho(y,t)dy,\quad
K''(x)=\psi(x).
\end{equation}
It is the \emph{aggregation equation} with a convex potential $K$ (as
$\psi\geq0$).

The global wellposedness of the aggregation equation has been
well-studied. A sharp Osgood condition has been derived in 
\cite{bertozzi2009blow, bertozzi2011lp, carrillo2011global}, which
distinguishes global regularity and finite time density concentration:
the solution is globally regular if and only if 
\begin{equation}\label{eq:Osgood}
\int_0^1\frac{1}{K'(r)}dr=\infty.
\end{equation}

For weakly singular interaction $\psi\sim r^{-s}$ with $s\in(0,1)$
near origin, or
more precisely, 
\begin{equation}\label{eq:psiweak}
  \lambda r^{-s}\leq \psi(r) \leq \Lambda r^{-s},\quad
  \Lambda\geq\lambda>0, \quad s\in(0,1),
\end{equation}
uniformly in $r\in(0,1]$,
the Osgood condition \eqref{eq:Osgood} is violated, and hence the
solution generates concentrations in finite time. The behavior is
different from the bounded interaction case $(s=0)$, in which \eqref{eq:Osgood}
holds.

In this paper, we study the global behavior of the Euler-Alignment
system with weakly singular interactions.

The following two theorems show a similar behavior to the system with
bounded interactions (Theorem \ref{thm:bounded}), for both
supercritical and subcritical regions of initial data. 
\begin{theorem}[Supercritical threshold condition]\label{thm:super}
Consider the 1D Euler-Alignment system \eqref{eq:G1D}-\eqref{eq:rho1D}
with smooth initial data $(\rho_0, G_0)$ and weakly singular
interaction $\psi$ satisfying \eqref{eq:psiweak}.

If $\displaystyle\inf_x G_0(x)<0$, then the solution admits a finite time blowup.
\end{theorem}

\begin{theorem}[Subcritical threshold condition]\label{thm:sub}
Consider the 1D Euler-Alignment system \eqref{eq:G1D}-\eqref{eq:rho1D}
with smooth initial data $(\rho_0, G_0)$ and weakly singular
interaction $\psi$ satisfying \eqref{eq:psiweak}.
 
If $\displaystyle\inf_x G_0(x)>0$, then there exists a globally regular
  solution.
\end{theorem}

The theorems imply that different behaviors between systems with
bounded and weakly singular interactions can only happen for 
critical initial data
\[\inf_xG_0(x)=0.\]
The example above ($G_0(x)\equiv0$) falls into this category. The
following theorem describes a large set of critical initial data, with
which the solution blows up in finite time.

\begin{theorem}[Blowup for critical initial data]\label{thm:cblowup}
  Consider the 1D Euler-Alignment system \eqref{eq:G1D}-\eqref{eq:rho1D}
with smooth initial data $(\rho_0, G_0)$ and weakly singular
interaction $\psi$ satisfying \eqref{eq:psiweak}.

If $G_0(x)\geq0$, and there exists an interval $I=[a,b]$ with $a<b$
such that for any $x\in I$, $G_0(x)=0$ and $\rho_0(x)>0$, then the
solution admits a finite time blowup.
\end{theorem}

The theorem says, if $G_0$ reaches zero in any non-vacuum interval, then
the solution will blow up in finite time. Very importantly, such
initial data will lead to a global smooth solution if the
communication weight is regular, due to Theorem \ref{thm:bounded}.
The different long time behaviors distinguish the two types of
interactions, and reveal the unique property of the weakly singular
interactions.

\medskip

The rest of the paper is organized as follows. In section
\ref{sec:local}, we develop a local wellposedness theory of the 1D
Euler-Alignment system, as well as a Beale-Kato-Majda criteria that
ensures the regularity.
Sections \ref{sec:super} and \ref{sec:sub} are devoted to prove
Theorem \ref{thm:super} and \ref{thm:sub}, respectively. A nonlinear
maximum principle is introduced to take care of the weak singularity
on the communication weight. The critical case is investigated in
section \ref{sec:critical}. We introduce a new proof for the blowup of
the aggregation equation. It utilizes local information and can be
extended to the Euler-Alignment system, proving Theorem
\ref{thm:cblowup}. Finally, in Section \ref{sec:multid}, we make
comments on the extension of our theory to higher dimensions.

\section{Local wellposedness and blowup criterion}\label{sec:local}
We start our discussion with a local wellposedness theory of our main
system in 1D. Recall the 1D Euler-Alignment system in $(\rho, G)$
representation
\begin{align}
&\pa_t\rho+\pa_x(\rho u)=0,\label{eq:rho}\\
&\pa_tG+\pa_x(G u)=0,\label{eq:G}\\
&\pa_xu=G-\psi\ast\rho, \label{eq:u}
\end{align}
where $\ast$ stands for convolution in $x$ variable.

\begin{theorem}[Local wellposedness]\label{thm:local}
Consider the 1D Euler-Alignment system \eqref{eq:rho}-\eqref{eq:u}
with smooth initial data with finite mass
$(\rho_0, G_0)\in (H^s\cap L_+^1)(\Omega)\times H^s(\Omega)$. 
Suppose the communication
weight is integrable:
\begin{equation}\label{eq:psiassumption}
\psi\in L^1(\Omega).
\end{equation}
Then, there exists a time $T>0$ such that the solution 
\[(\rho,G)\in \mathcal{C}([0,T]; (H^s\cap L^1_+)(\Omega))\times \mathcal{C}([0,T];
H^s(\Omega)).\]
Moreover, the solution stays smooth up to time $T$ as long as 
\begin{equation}\label{eq:BKM}
\int_0^T\left(\|\rho(\cdot,t)\|_{L^\infty}+\|G(\cdot,t)\|_{L^\infty}\right)dt<+\infty.
\end{equation}
\end{theorem} 

\begin{proof}
We first state an $H^s$-estimate on $\rho$
\[\frac{d}{dt}\|\rho(\cdot,t)\|_{H^s}^2\lesssim\left[\|\rho\|_{L^\infty}+\|\pa_xu\|_{L^\infty}\right]\left(\|\rho\|_{H^s}^2+\|\pa_xu\|_{H^s}^2\right).\]
The proof can be found, for instance, in \cite[Theorem Appendix
A.2]{carrillo2016critical}.

As $G$ satisfies the same continuity equation as $\rho$, we have
\begin{equation}\label{eq:GHs}
\frac{d}{dt}\|G(\cdot,t)\|_{H^s}^2\lesssim\left[\|G\|_{L^\infty}+\|\pa_xu\|_{L^\infty}\right]\left(\|G\|_{H^s}^2+\|\pa_xu\|_{H^s}^2\right).
\end{equation}

Putting these two estimate together, we obtain
\[\frac{d}{dt}\big(\|\rho(\cdot,t)\|_{H^s}^2+\|G(\cdot,t)\|_{H^s}^2\big)\lesssim
\left[\|\rho\|_{L^\infty}+\|G\|_{L^\infty}+\|\pa_xu\|_{L^\infty}\right]
\left(\|\rho\|_{H^s}^2+\|G\|_{H^s}^2+\|\pa_xu\|_{H^s}^2\right).\]

From the relation \eqref{eq:u}, we can estimate $\pa_xu$ by $\rho$
and $G$ as follows. For a fixed time $t$,
\begin{align*}
& \|\pa_xu\|_{L^\infty}\leq\|\psi\ast\rho\|_{L^\infty}+\|G\|_{L^\infty}\leq
\|\psi\|_{L^1}\|\rho\|_{L^\infty}+\|G\|_{L^\infty},\\
& \|\pa_xu\|_{H^s}\leq\|\psi\ast\rho\|_{H^s}+\|G\|_{H^s}\leq
\|\psi\|_{L^1}\|\rho\|_{H^s}+\|G\|_{H^s}.
\end{align*}

Since $\|\psi\|_{L^1}$ is bounded, we now arrive at the estimate
\[\frac{d}{dt}\big(\|\rho(\cdot,t)\|_{H^s}^2+\|G(\cdot,t)\|_{H^s}^2\big)\lesssim
\left[\|\rho(\cdot,t)\|_{L^\infty}+\|G(\cdot,t)\|_{L^\infty}\right]
\left(\|\rho(\cdot,t)\|_{H^s}^2+\|G(\cdot,t)\|_{H^s}^2\right).\]
Standard Gronwall inequality implies
\[\|\rho(\cdot,t)\|_{H^s}^2+\|G(\cdot,t)\|_{H^s}^2\leq
\big(\|\rho_0\|_{H^s}^2+\|G_0\|_{H^s}^2\big)\exp
\left[\int_0^t \|\rho(\cdot,s)\|_{L^\infty}+\|G(\cdot,s)\|_{L^\infty}
  ds\right].\]

Therefore, if condition \eqref{eq:BKM} is satisfied, 
$\rho(\cdot,t), G(\cdot,t)\in H^s(\R)$ for all $t\in[0,T]$. This ends
the proof of the theorem.
\end{proof}

We shall make several remarks regarding Theorem \ref{thm:local}.
\begin{remark}
A local wellposedness proof for 1D Euler-Alignment system has been
done in \cite[Theorem Appendix A.1]{carrillo2016critical}, with an
additional assumption on $\psi$
\[\int_\Omega x\psi'(x)dx<+\infty.\]
Here, we relax the assumption by making use of the $(\rho,G)$
formulation of the system. 

If assumption \eqref{eq:psiassumption} is violated, namely $\psi$ is
not integrable at the origin, then the behavior of the equation
changes dramatically due to the strongly singular interaction. We
refer to \cite{do2018global, kiselev2018global} for discussions on
local and global regularities under such setup.
\end{remark}

\begin{remark}
  Condition \eqref{eq:BKM} is called the Beale-Kato-Majda (BKM) type
  criteria. It provides a sufficient and necessary condition under which the
  solution stays smooth. Condition \eqref{eq:BKM} is equivalent to 
\begin{equation}\label{eq:BKMux}
\int_0^T\|\pa_xu(\cdot,t)\|_{L^\infty}dt<+\infty,
\end{equation}
which is a standard sufficient condition to ensure the wellposedness of
  the characteristic paths for presureless Euler dynamics.
The equivalency is due to the following estimates
  \begin{align*}
    &\int_0^T\|\pa_xu(\cdot,t)\|_{L^\infty}dt\leq
  \int_0^T\left(\|\rho(\cdot,t)\|_{L^\infty}+\|\psi\|_{L^1}\|G(\cdot,t)\|_{L^\infty}\right)dt,\\
    &\|\rho(\cdot,t)\|_{L^\infty}+\|G(\cdot,t)\|_{L^\infty}\leq(\|\rho_0\|_{L^\infty}+\|G_0\|_{L^\infty})\exp\left[\int_0^t\|\pa_xu(\cdot,s)\|_{L^\infty}ds\right].
  \end{align*}
\end{remark}

\begin{remark}\label{rem:addconst}
When $\Omega=\R$, assumption \eqref{eq:psiassumption} can be further
generalized to
\[\psi\in L^1(\R)+\textnormal{const}.\]
This allows us to include more types of communication weight, for
instance $\psi\equiv1$. We include a short proof for the sake of completeness.
\end{remark}

\begin{proof}[Proof of Remark \ref{rem:addconst}]
Let $\psi=\psi_0+c$, where $\psi_0\in L^1(\R)$ and $c$ is a
constant. Define $G=\pa_xu+\psi_0\ast\rho$. Then, the $(\rho, G)$
representation of the 1D Euler-Alignment system reads
\[\pa_t\rho+\pa_x(\rho u)=0,\quad
\pa_tG+\pa_x(G u)=-cm\pa_xu,\quad
\pa_xu = G-\psi_0\ast\rho,\]
where $m=\int_\Omega\rho_0(x)dx$ is the total mass which is preserved in
time.

Due to the extra term in the dynamics of $G$, the $H^s$ estimate on
$G$ \eqref{eq:GHs} becomes
\begin{align*}
\frac{d}{dt}\|G(\cdot,t)\|_{H^s}^2\lesssim&\left[\|G\|_{L^\infty}+\|\pa_xu\|_{L^\infty}\right]\left(\|G\|_{H^s}^2+\|\pa_xu\|_{H^s}^2\right)+cm\|G\|_{H^s}\|\pa_xu\|_{H^s}\\
\lesssim&\left[1+\|G\|_{L^\infty}+\|\pa_xu\|_{L^\infty}\right]\left(\|G\|_{H^s}^2+\|\pa_xu\|_{H^s}^2\right).
\end{align*}

The rest of the proof stays the same as Theorem \ref{thm:local}.
\end{proof}

A natural question would be, whether the BKM criteria \eqref{eq:BKM}
can be further reduced to
\begin{equation}\label{eq:BKM2}
\int_0^T\|G(\cdot,t)\|_{L^\infty}dt<+\infty.
\end{equation}
In another word, whether boundedness of $G$ implies boundedness of
$\rho$. If so, global regularity of the system becomes equivalent to the
boundedness of $G$.

The following proposition shows that condition \eqref{eq:BKM2} indeed
serves as a BKM criterion for the 1D Euler-Alignment system, when the
communication weight is bounded.

\begin{proposition}[An enhanced BKM criterion for system with bounded
  interactions]
\label{prop:BKMplus}
Consider the initial value problem of the 1D Euler-Alignment system
\eqref{eq:rho}-\eqref{eq:u} with smooth initial data $(\rho_0,G_0)\in
(H^s\cap L^1_+)(\Omega)\times H^s(\Omega)$.
Suppose the communication weight is bounded
and integrable:
\[\psi\in (L^1\cap L^\infty)(\Omega)+\textnormal{const}.\]
Suppose criteria \eqref{eq:BKM2} is satisfied for time $T$.
Then, the solution is smooth up to time $T$, namely
\[(\rho,G)\in \mathcal{C}([0,T]; (H^s\cap L^1_+)(\Omega))\times \mathcal{C}([0,T];
H^s(\Omega)).\]
\end{proposition}
\begin{proof}
It suffies to prove that \eqref{eq:BKM2} implies \eqref{eq:BKM}.

Consider the characteristic path $X(t):=X(t; x_0)$ starting at $x_0\in\Omega$
\[\frac{d}{dt}X(t; x_0)=u(X(t; x_0),t),\quad X(0; x_0)=x_0.\]
As $\rho$ satisfies the continuity equation \eqref{eq:rho}, we get
\[\frac{d}{dt}\rho(X(t),t)=-\pa_xu(X(t),t)\rho(X(t),t).\]
Then,
\[\rho(X(t),t)=\rho_0(x)\exp\left[-\int_0^t \pa_xu(X(s),s) ds\right]
\leq\rho_0(x)\exp\left[\int_0^t\|\pa_xu(\cdot,s)\|_{L^\infty}ds\right].\]
Since $\psi$ is bounded, we can estimate
\begin{equation}\label{eq:BKMest}
\|\pa_xu(\cdot,t)\|_{L^\infty}\leq\|G(\cdot,t)\|_{L^\infty}+m\|\psi\|_{L^\infty}.
\end{equation}
Therefore, we get
\[\|\rho(\cdot,t)\|_{L^\infty}\leq\|\rho_0\|_{L^\infty}\exp\left[
m\|\psi\|_{L^\infty}t+\int_0^t\|G(\cdot,t)\|_{L^\infty}ds\right].\]
Hence, the boundedness of $G$ does imply the boundedness of $\rho$.
\end{proof}

Using Proposition \ref{prop:BKMplus}, one can easily prove Theorem
\ref{thm:bounded}, by showing criterion \eqref{eq:BKM2} is satisfied
if and only if $\inf_xG_0(x)\geq0$. We refer readers to
\cite{carrillo2016critical} for details.

When the communication weight is weakly singular, Proposition
\ref{prop:BKMplus} might be false. In particular, the estimate
\eqref{eq:BKMest} is no longer available. One alternative bound could
be
\[\|\pa_xu(\cdot,t)\|_{L^\infty}\leq\|G(\cdot,t)\|_{L^\infty}
+\|\rho(\cdot,t)\|_{L^\infty}\|\psi\|_{L^1}.\]
It implies an implicit bound
\[\|\rho(\cdot,t)\|_{L^\infty}\leq\|\rho_0\|_{L^\infty}\exp\left[
m\|\psi\|_{L^1}\int_0^t\|\rho(\cdot,s)\|_{L^\infty}ds+\int_0^t\|G(\cdot,t)\|_{L^\infty}ds\right],\]
which is not enough to obtain boundedness of $\rho$.

In fact, a counter example such that Proposition \ref{prop:BKMplus}
fails for weakly singular interaction has been mentioned in the
introduction, where $G_0(x)\equiv0$. The corresponding aggregation
system \eqref{eq:agg} is known to have a finite time loss of
regularity as long as $\psi$ is unbounded at the origin.
Therefore, the global regularity theory of Euler-Alignment system with
bounded interaction can not be directly extended to the case when the
interaction is weakly singular.

\section{Supercritical threshold condition}\label{sec:super}
\subsection{Finite time blowup on $G$}
In this section, we prove Theorem \ref{thm:super}: solution forms a
singularity in finite time, for supercritical initial data
\[\inf_{x\in\Omega} G_0(x)<0.\]

Under such configuration, there exists an $x_0\in\Omega$ such that
$G_0(x_0)<0$. Denote $X(t)$ be the characteristic path starting at $x_0$
\[\frac{d}{dt}X(t)=u(X(t),t),\quad X(0)=x_0.\]
As long as the solution stays smooth, alongside $X(t)$, we have
\[\frac{d}{dt}G(X(t),t)=-\pa_xu(X(t),t)G(X(t),t).\]
This implies
\[G(X(t),t)=G_0(x_0)\exp\left[\int_0^t\pa_xu(X(s),s)ds\right]<0.\]
Moreover, $\psi\ast\rho(\cdot,t)\geq0$ for any $t\geq0$.
From \eqref{eq:u}, we get
\[\frac{d}{dt}G(X(t),t)=-G^2(X(t),t)+G(X(t),t) 
\big(\psi\ast\rho(\cdot,t)\big)(X(t))\leq -G^2(X(t),t).\]
Applying a classical comparison principle, we obtain
\[G(X(t),t)\leq \frac{1}{t+\frac{1}{G_0(x_0)}}
\xrightarrow{~~t\to-\frac{1}{G_0(x_0)}~~}-\infty.\]
Therefore, there exists a finite time $T\leq-\frac{1}{G_0(x_0)}$, such
that 
\begin{equation}\label{eq:Gblowup}
\lim_{t\to T-}G(X(t),t)=-\infty.
\end{equation}
The BKM criterion \eqref{eq:BKM} fails at time $T$, which leads to a
loss of regularity.

Now, we discuss the behavior of the solution $(\rho, u)$ at the blowup
time $T$.

\begin{lemma}\label{lem:shock}
Let $T$ be the time that the first blowup of $G$ occurs, and the
corresponding location is $x=X(T; x_0)$. Suppose the solution $(\rho,
u)$ stays smooth for $t\in[0,T)$.
Then, the solution develops a shock at time $T$ and
location $x$, namely
\begin{equation}\label{eq:shock}
\lim_{t\to T}\pa_xu(X(t; x_0),t)=-\infty.
\end{equation}
Moreover, if $\rho_0(x_0)>0$, then the density concentrates at the
shock location (called \emph{singular shock})
\begin{equation}\label{eq:densityconcentration}
\lim_{t\to T}\rho(X(t; x_0),t)=+\infty.
\end{equation}
\end{lemma}
\begin{proof}
From \eqref{eq:u}, we know $\pa_xu\leq G$. This together with
\eqref{eq:Gblowup} implies shock formation
\[\pa_xu(X(t; x_0),t)\leq G(X(t; x_0),t)\xrightarrow{~t\to T~}-\infty.\]

Define $F=G/\rho$. Then, $F$ satisfies the local transport equation
\[\pa_tF+u\pa_xF=0.\]

Since $\rho_0(x_0)>0$, $F_0$ is bounded and smooth in a
neighborhood of $x_0$. Then, $F$ is well-defined alongside the
characteristic path $X(t; x_0)$, and
\[F(X(t; x_0),t)=F_0(x_0).\]
Therefore, we obtain a concentration of density 
\[\rho(X(t; x_0),t)= \frac{\rho_0(x_0)}{G_0(x_0)}G(X(t;x_0),t)\xrightarrow{t\to T}+\infty.\]
\end{proof}

Lemma \ref{lem:shock} does not rule out the possibility that blowup
happens before $G$ becomes singular.
Indeed, the BKM criterion \eqref{eq:BKM} could fail if $\rho$ becomes
unbounded.

We now construct an example when $\rho$ blows up before
$G$. This would imply that the criterion \eqref{eq:BKM2} itself does
not guarantee regularity of the system. So Proposition
\ref{prop:BKMplus} is no longer true for the system with weakly
singular interactions.

\subsection{An example: $\rho$ blows up before $G$}
Take $\Omega=\R$. Let $\rho_0$ be a smooth function 
supported in $(0,1)$.

Let $\eta$ be a smooth function such that $\eta\geq0$,
$\max_x\eta(x)=1$, and $\text{supp}\eta=(0,1)$.
Consider the following $G_0$
\begin{equation}\label{eq:Gzero}
  G_0(x)=-\epsilon\eta(x-L),
\end{equation}
where $L>0$ is a large number, and $\epsilon>0$ is a small positive
number to be chosen.
As $\inf G_0(x)=-\epsilon<0$, $G_0$ is a supercritical initial condition.

Note that $\text{supp}(G_0)=(L, L+1)$. If $L$ is large enough,
$\text{supp}(\rho_0)\cap\text{supp}(G_0)=\emptyset$.
Starting from any $x_0\in\text{supp}(\rho_0)$, we have
$G_0(x_0)=0$ and consequently $G(X(t;x_0),t)=0$.
So, $\pa_xu=-\psi\ast\rho$ in the support of $\rho$, since
Since $\pa_xu$ is locally depended on $G$.
Therefore, the dynamics of $\rho$ does not depend on $G$, and it is
the same as the aggregation equation \eqref{eq:agg}. Since $\psi$ is
singular, we know the density $\rho$ concentrates at a finite time
$T_*$, which is independent of $L$ and $\epsilon$.

The goal is to show $G$ remains regular at time $T_*$. It suffies to
prove that $G$ is bounded from below at $T_*$. To this end, we shall
obtain a lower bound estimate on $G$. Fix $x_0\in\text{supp}(G_0)$. 
Then,  along its characteristic path, we have
\begin{equation}\label{eq:Gchar}
  \frac{d}{dt}G(X(t;x_0),t)=-G(X(t;x_0),t)\pa_xu(X(t;x_0),t)=-G^2+G~\psi\ast\rho.
\end{equation}
As $G(X(t;x_0),t)<0$, we need an upper bound on $\psi\ast\rho$.

If the supports of $\rho(\cdot,t)$ and $G(\cdot,t)$ are
well-separated, namely
\begin{equation}\label{eq:suppsep}
  \text{dist}(\text{supp}(\rho(\cdot,t)), \text{supp}(G(\cdot,t)))\geq1,
\end{equation}
for $t\in[0,T_*]$, then we have the estimate
\[\psi\ast\rho (X(t;x_0),t)=\int_{\text{supp}\rho(t)}\psi(X(t;x_0)-y,t)\rho(y,t)dy
  \leq\psi(1)\int_{\text{supp}\rho(t)}\rho(y,t)dy=\psi(1)m.\]
Let us denote the constant $C=\psi(1)m$. It is uniform in $x_0\in\text{supp}(G_0)$
and $t\in[0,T_*]$. Apply the estimate to \eqref{eq:Gchar}, we get
\[\frac{d}{dt}G(X(t;x_0),t)\geq-G^2-CG.\]
An explicit calculation yields
\[G(X(t;x_0),t)\geq-\frac{C}{\frac{G_0(x_0)-c}{G_0(x_0)}e^{-Ct}-1}.\]
So, if $G_0(x_0)\geq-C$, then
\[G(X(t;x_0),t)\geq -C,\quad\forall~t\in\left[~0,~ \frac{1}{C}\ln\left(\frac{C-G_0(x_0)}{-2G_0(x_0)}\right)~\right].\]

Note that
\[\lim_{z\to0-}\left[\frac{1}{C}\ln\left(\frac{C-z}{-2z}\right)\right]=+\infty.\]
It means that if we pick $\epsilon$ small enough, $G(\cdot,t)$ can be
bounded below by $-C$ for a sufficiently long time. In particular, we can choose
$\epsilon$ small enough, e.g.
\[\epsilon=\frac{C}{2e^{CT_*}-1},\]
so that $G$ is bounded until $t=T_*$.

It remains to show that condition \eqref{eq:suppsep} holds at
$t\in[0,T_*]$.

One important feature of the Euler-Alignment system
\eqref{eq:EArho}-\eqref{eq:EAu} is that the velocity is uniformly
bounded in time
\begin{equation}\label{eq:maxprin}
  \|u(\cdot,t)\|_{L^\infty}\leq\|u_0\|_{L^\infty}.
\end{equation}
Indeed, a maximum principle can be easily derived from \eqref{eq:EAu} (see for
instance \cite{tadmor2014critical}). Moreover, under additional
assumptions, not only boundedness but also contraction on $u$ can be
proved, which reveals the so-called \emph{flocking phenomenon}.

Take $x_1\in\text{supp}(\rho_0)$ and $x_2\in\text{supp}(G_0)$. Then,
\[\frac{d}{dt}\left(X(t;x_2)-X(t;x_1)\right)=u(X(t;x_2),t)-u(X(t;x_1),t)
  \geq-2\|u(\cdot,t\|_{L^\infty}\geq-2\|u_0\|_{L^\infty}.\]
Hence,
\[X(t;x_2)-X(t;x_1)\geq (x_2-x_1)-2t\|u_0\|_{L^\infty}\geq
  (L-1)-2t\|u_0\|_{L^\infty}.\]
If we take $L$ big enough (e.g. $L=2+2T_*\|u_0\|_{L^\infty}$), then
the distance will remain big at time $T_*$. Therefore,
\eqref{eq:suppsep} holds for $t\in[0,T_*]$.

\subsection{The BKM criterion}\label{sec:BKM}
The example above states that $\rho$ could blow up before $G$. On the
other hand, $G$ could blow up before $\rho$ as well. Examples can
be constructed similarly, by letting $\epsilon$ in \eqref{eq:Gzero}
large.

Therefore, both terms in the BKM criterion \eqref{eq:BKM} are
necessary to ensure regularity. This is very different from the system
with bounded interactions.

\section{Subcritical theshold condition}\label{sec:sub}
In this section, we turn to study the Euler-Alignment system with
weakly singular interactions, for subcritical initial data
\begin{equation}\label{eq:subG}\inf_{x\in\Omega}G_0(x)>0.\end{equation}
Since $G$ satisfies the continuity equation \eqref{eq:G}, it is easy
to show that positivity preserves in time, namely
\[G(x,t)>0,\quad \forall~x\in\Omega,~~t\geq0.\]
Hence, the blowup \eqref{eq:Gblowup} can not happen.
However, unlike the case with bounded interactions, the boundedness of
$G$ (criterion \eqref{eq:BKM2}) is not enough to ensure global
regularity, as argued in Section \ref{sec:BKM}.
In order to prove Theorem \ref{thm:sub}, we need to obtain bounds on
both $G$ and $\rho$.

\subsection{A global estimate on $\rho$}
We start with an estimate on $\rho$. Along the characteristic path, we
have
\[\frac{d}{dt}\rho(X(t),t)=-\rho(X(t),t)\pa_xu(X(t),t)
=-\rho~G+\rho~\psi\ast\rho.\]
The first term on the right hand side is a good term that helps bring
down the value of $\rho$ alongside the characteristic path, while the
second term is a bad term.

\subsubsection*{Step 1: An estimate on the good term}
Let $q = \rho/G = 1/F$. Then, $q$ satisfies the transport equation
\[\pa_tq+u\pa_xq=0.\]
Since $G_0$ satisfies \eqref{eq:subG}, $q_0$ is bounded and smooth.
Clearly, we have
\[q(X(t;x),t)=q_0(x).\]
Therefore, we obtain a lower bound estimate on $G$
\[G(X(t;x),t)=\frac{\rho(X(t;x),t)}{q(X(t;x),t)}=\frac{\rho(X(t;x),t)}{q_0(x)}
\geq \frac{\rho(X(t;x),t)}{\|q_0\|_{L^\infty}}.\]
It yields an estimate on the good term
\begin{equation}\label{eq:goodestimate}
-\rho G\leq-C_1\rho^2,
\end{equation}
where the constant $C_1=1/\|q_0\|_{L^\infty}$ is bounded and depend only
on the initial data.

\subsubsection*{Step 2: An estimate on the bad term}
To estimate the bad term, and to compare with the good term, we need
a local bound on $\psi\ast\rho$. 

A nonlinear maximum principle is introduced in
\cite{constantin2012nonlinear} which offers a local bound, at the
extrema of $\rho$, when $\psi$ is strongly singular. Here, we state
a lemma which serves as a nonlinear maximum principle for weakly
singular kernel.

\begin{lemma}[Nonlinear maximum principle]\label{lem:NMP}
Let $\psi$ be a weakly singular communication weight satisfying
condition \eqref{eq:psiweak}.
Consider a function $f\in L_+^1(\R)$ and a point $x_*$ such that 
$f(x_*)=\max f(x)$. Then, there exists a constant $C>0$, depending on
$\Lambda$, $s$ and $\|f\|_{L^1}$, such that
\begin{equation}\label{eq:NMP}
\psi\ast f(x_*)\leq Cf(x_*)^s.
\end{equation}
\end{lemma}
\begin{proof}
First of all, since $\psi\ast f\leq (\Lambda |x|^{-s})\ast f$, it suffies
to prove \eqref{eq:NMP} for $\psi(x)=|x|^{-s}$.
For any $a>0$, we compute
\begin{align*}
&\big( (|x|^{-s})\ast f\big)(x_*)=\int_{|y|\leq a}f(x_*-y)|y|^{-s}dy+
\int_{|y|>a}f(x_*-y)|y|^{-s}dy\\
&=f(x_*)\int_{|y|\leq a}|y|^{-s}dy
-\int_{|y|\leq a}(f(x_*)-f(x_*-y))|y|^{-s}dy+\int_{|y|>a}f(x_*-y)|y|^{-s}dy\\
&\leq\frac{2a^{1-s}}{1-s}f(x_*)
-a^{-s}\int_{|y|\leq a}(f(x_*)-f(x_*-y))dy+a^{-s}\int_{|y|>a}f(x_*-y)dy\\
&=\frac{2a^{1-s}}{1-s}f(x_*)-2a^{1-s}f(x_*)+a^{-s}\|f\|_{L^1}
=\frac{2s}{1-s}a^{1-s}f(x_*)+a^{-s}\|f\|_{L^1}.
\end{align*}
Take $a=\|f\|_{L^1}/(2f(x_*))$, we obtain
\[\big( (|x|^{-s})\ast f\big)(x_*)\leq\left(\frac{2-s}{1-s}2^s\|f\|_{L^1}^{1-s}\right) f(x_*)^s.
\]
\end{proof}

We now apply Lemma \ref{lem:NMP} with $f=\rho(\cdot,t)$. Fix any time
$t$, and let $x_*$ be the location where maximum of $\rho(\cdot,t)$ is
attained. Then,
\begin{equation}\label{eq:badestimate}
\rho(x_*,t)\psi\ast\rho(x_*,t)\leq C_2\rho(x_*,t)^{1+s},
\end{equation}
where the constant $C_2=C_2(\Lambda,s,m)>0$.

\subsubsection*{Step 3: a uniform upper bound on $\rho$}
Combining the two estimates \eqref{eq:goodestimate} and
\eqref{eq:badestimate}, we obtain that if $x_*$ is a point such
that $\rho(x_*,t)=\max_x\rho(x,t)$, then
\[\pa_t\rho(x_*,t)\leq C_1\rho(x_*,t)^2-C_2\rho(x_*,t)^{1+s}.\]
So, when $\rho$ is large enough such taht $\rho\geq(C_2/C_1)^{1/(1-s)}$,
then $\pa_t\rho(x_*,t)\leq0$. Therefore, we
obtain an apriori bound on the density
\[\|\rho(\cdot,t)\|_{L^\infty}\leq 
\max\left\{\|\rho_0\|_{L^\infty}, ( C_2/C_1)^{1/(1-s)}\right\}=:C_\rho,
\quad\forall~t\geq0.\]

Note that the bound $C_\rho=C_\rho(\Lambda, s, m,
\|\rho_0\|_{L^\infty},\|q_0\|_{L^\infty})$ is independent of
time. Therefore,
$\|\rho(\cdot,t)\|_{L^\infty}$ is uniformly bounded.

\subsection{An estimate on $G$}
We are left to bound $G$, which is not hard to obtain given the
apriori estimate on $\rho$.

Along the characteristic path, $G$ satisfies \eqref{eq:Gchar}. Recall
\[\frac{d}{dt}G(X(t),t)=-G^2+G~\psi\ast\rho.\]

The uniform bound on $\rho$ implies a bound on $\psi\ast\rho$
\[\|\psi\ast\rho\|_{L^\infty}\leq\|\psi\|_{L^1}\|\rho\|_{L^\infty}\leq\|\psi\|_{L^1}C_\rho,\quad\forall~t\geq0.\]
Therefore, we obtain
\[\frac{d}{dt}G(X(t),t)\leq-G(G-\|\psi\|_{L^1}C_\rho).\]
So, $G$ can not grow along the characteristic path if
$G\geq\|\psi\|_{L^1}C_\rho$.
It yields a uniform bound on $G$
\[\|G(\cdot,t)\|_{L^\infty}\leq\left\{\|G_0\|_{L^\infty},~\|\psi\|_{L^1}C_\rho\right\}.\]

\section{The critical case}\label{sec:critical}
This section is devoted to discuss the critical case, when the initial
condition satisfies
\begin{equation}\label{eq:Gcritical}
\min_{x\in\R}G_0(x)=0.
\end{equation}

Theorems \ref{thm:super} and \ref{thm:sub} show that, the behavior of the
Euler-Alignment system with weakly singular interactions is the same
as the system with bounded interactions, in both subcritical and
supercritical regimes.
Therefore, different behaviors can only happen in the critical senario.

One special critical initial condition is $G_0(x)\equiv0$. The system
reduces to the aggregation equation \eqref{eq:agg}. As we have argued
in the introduction, the weakly singular interaction will drive the
solution towards a finite time blowup. Therefore, Theorem
\ref{thm:bounded} no longer holds.

A natural question arises: what happens for general critical initial
data?

Recall the dynamics of the density $\rho$ along the characteristic
path
\[\frac{d}{dt}\rho(X(t),t)=-\rho~G+\rho~\psi\ast\rho.\]
If $G_0(x_0)=0$, then from \eqref{eq:Gchar} we have
$G(X(t),t)=0$. Therefore, the good term $-\rho G$ turns off near
$X(t)$, and the local behavior of dynamics becomes the same as the
aggregation equation.

To capture such behavior, we shall first provide an alternative proof
to the blowup of the aggregation equation. Unlike
\cite{bertozzi2009blow}, the proof traces local information along the
characteristic paths. The idea is partly inspired from \cite{choi2015finite}.

\subsection{A ``local'' proof for blowup of the aggregation equation}\label{sec:blowupagg}
Let us consider the aggregation equation in the form
\[\pa_t\rho+\pa_x(\rho u)=0,\quad \pa_xu=-\psi\ast\rho.\]

Without loss of generality, we assume that $\rho_0$ is strictly
positive in some interval $I=[a,b]$, namely
\begin{equation}\label{eq:rhopositive}
  \rho_0(x)\geq c>0,\quad\forall~x\in[a,b].
\end{equation}

Let $r(t)=X(t;b)-X(t;a)$. Then,

\begin{align*}
\frac{d}{dt}r(t)=&~u(X(t;b),t)-u(X(t;a),t)
=\int_{X(t;a)}^{X(t;b)}\pa_xu(y,t)dy\\
=&-\int_{X(t;a)}^{X(t;b)}(\psi\ast\rho)(y,t)dy
=-\int_{X(t;a)}^{X(t;b)}\int_{-\infty}^{\infty}\psi(y-z)\rho(z,t)dy\\
\leq&-\int_{X(t;a)}^{X(t;b)}\int_{X(t;a)}^{X(t;b)}\psi(y-z)\rho(z,t)dzdy.
\end{align*}
For $y, z\in[X(t;a), X(t;b)]$, we have $|y-z|\leq2r(t)$.
By weakly singular condition \eqref{eq:psiweak}, this implies
$\psi(y-z)\geq\lambda \big(2r(t)\big)^{-s}.$
Therefore, we obtain
\begin{equation}\label{eq:estrt}
  \frac{d}{dt}r(t)\leq-\lambda\big(2r(t)\big)^{-s} r(t)
  \int_{X(t;a)}^{X(t;b)}\rho(z,t)dz.
\end{equation}
The following lemma shows a local conservation of mass along
characteristic paths.
\begin{lemma}[Conservation of mass]\label{lem:masscons}
  Let $\rho$ be a strong solution of the continuity equation
  \[\pa_t\rho+\pa_x(\rho u)=0.\]
  Let $X(t;x_1), X(t,x_2)$ be two characteristic paths starting at
  $x_1$ and $x_2$, respectively. Then,
  \begin{equation}\label{eq:masscons}
    \int_{X(t;x_1)}^{X(t;x_2)}\rho(x,t)dx=\int_{x_1}^{x_2}\rho_0(x)dx,\quad\forall~t\geq0.
  \end{equation}
  Namely, the mass in the interval $[X(t;x_1), X(t;x_2)]$ is conserved
  in time.
\end{lemma}
\begin{proof}
  Compute
  \begin{align*}
    \frac{d}{dt}&\int_{X(t;x_1)}^{X(t;x_2)}\rho(x,t)dx\\
    &=\rho(X(t;x_2),t)\frac{d}{dt}X(t;x_2)-\rho(X(t;x_1),t)\frac{d}{dt}X(t;x_1)+                                \int_{X(t;x_1)}^{X(t;x_2)}\pa_t\rho(x,t)dx\\                    
                &=\rho(X(t;x_2),t)u(X(t;x_2),t)-\rho(X(t;x_1),t)u(X(t;x_1),t)+
                  \int_{X(t;x_1)}^{X(t;x_2)}\pa_t\rho(x,t)dx\\                    
                &=\int_{X(t;x_1)}^{X(t;x_2)}\pa_x\left(\rho(x,t)u(x,t)\right)dx+
                  \int_{X(t;x_1)}^{X(t;x_2)}\pa_t\rho(x,t)dx=0.                    
  \end{align*}
  This directly implies the conservation of mass \eqref{eq:masscons}.
\end{proof}

Applying Lemma \ref{lem:masscons} to \eqref{eq:estrt} with $x_1=a$ and
$x_2=b$, and using  the
lower bound assumption \eqref{eq:rhopositive}, we obtain
\[\frac{d}{dt}r(t)\leq- 2^{-s}\lambda
                      \big(r(t)\big)^{1-s}\int_a^b\rho_0(z)dz
\leq-2^{-s}c(b-a)\lambda\big(r(t)\big)^{1-s}.\]

For $s\in(0,1)$, it is easy to show that $r(t)$ reaches zero in finite
time. Indeed, a standard comparison principle yields
\[r(t)\leq\left[(b-a)^s-2^{-s}c(b-a)\lambda st\right]^{1/s},\]
where the right hand side touches zero at
\[T_*=\frac{2^s}{c(b-a)^{1-s}\lambda s}<\infty.\]
Then, $r(t)$ should reach zero no later than $T_*$.

The quantity $r(t)=0$ means that two characteristic paths run into each other. It
indicates a shock formation with $\pa_xu(x,t)\to-\infty$. Therefore,
the solution loses regularity in finite time.

\subsection{Finite time blowup for a class of critical initial data}
Now, let us consider the Euler-Alignment system
\eqref{eq:rho}-\eqref{eq:u} with critical initial data
\eqref{eq:Gcritical}.

Suppose there exists an interval $I=[a,b]$ such that $\rho_0$ is
strictly positive, and $G_0$ is zero, namely
\begin{equation}\label{eq:critGblow}
  \forall~x\in [a,b], \quad \rho_0(x)\geq c>0, \quad G_0(x)=0.
\end{equation}
Then, from \eqref{eq:Gchar} we obtain
\[G(x,t)=0,\quad \forall~t\geq0,\quad x\in[X(t;a), X(t;b)].\]
Therefore, the dynamics of $\rho$ between the two characteristic paths
$X(t;a)$ and $X(t;b)$ should be the same as the dynamics of the
corresponding aggregation equation, as long as the solution stays
smooth. The blowup estimates for the aggregation equation in Section
\ref{sec:blowupagg} can be directly applied to the Euler-Alignment
system. Therefore, the same type of blowup as the aggregation equation
happens in finite time. This ends the proof of Theorem \ref{thm:cblowup}.

\begin{remark}
The condition \eqref{eq:critGblow} contains a large family of critical
initial data, under which the global behaviors of the Euler-Alignment
system is different between the bounded and weakly singular
interactions. The condition is sharp in the following sense.

Consider the following critical initial data $(\rho_0, G_0)$:
\begin{equation}\label{eq:critGreg}
\exists~ q_0\in L^\infty,\quad \text{such that}\quad
\rho_0(x)=q_0(x)G_0(x).
\end{equation}
Then, the arguments in section \ref{sec:sub} can be easily extended,
allowing $G_0(x)=0$. Hence, the solution exists globally in time.

Note that condition \eqref{eq:critGreg} implies that $G_0(x)=0$ only
occurs at $\rho_0(x)=0$. which is almost the opposite of condition
\eqref{eq:critGblow}. Therefore, condition \eqref{eq:critGblow} is a
sharp condition so that the global behaviors of systems with bounded and weakly
singular interactions are different from each other.

Rare exceptions could happen. For instance, 
$G_0(x)=0$ only at a single point $x_0$, with $\rho_0(x_0)>0$.
It satisfies neither \eqref{eq:critGblow} nor \eqref{eq:critGreg}.
In this case, a more subtle ``local'' proof is required for the
corresponding aggregation system in order to obtain a finite time
blowup. This will be left for further investigations.
\end{remark}

\section{Extensions to higher dimensions}\label{sec:multid}
The global behaviors of the Euler-Alignment system
\eqref{eq:EArho}-\eqref{eq:EAu} is much less understood in higher
dimensions. With bounded interactions, the system was first studied in
\cite{tadmor2014critical} in two dimensions. Threshold conditions on
initial data were obtained, but the result was not sharp.

The $G$ quantity can be defined as $G=\div u+\psi\ast\rho$. However,
it does not satisfy the continuity equation any more. The dynamics of
$G$ reads
\[\pa_tG+\grad\cdot(Gu)=\text{tr}(\grad u^{\otimes2})-(\div u)^2.\]
The right hand side is called the \emph{spectral gap}, which is
generally non-zero in two or higher dimensions. 

The system in $(\rho, G)$ formulation in 2D has been studied in
\cite{he2017global}, where improved threshold conditions are
obtained. However, the result is still far from being sharp, due to
the lack of control in the spectral gap.

For the Euler-Alignment system with weakly singular interactions, our
arguments on local wellposedness (Section \ref{sec:local}) as well as both
supercritical and subcritical threshold conditions (Sections
\ref{sec:super} and \ref{sec:sub}) can be extended to higher
dimensions, using similar techniques to
handle the spectral gap. However, we are not able to distinguish the
behaviors between the systems with bounded and weakly
singular interactions (Section \ref{sec:critical}) until we get a sharp
threshold condition.

\bigskip\noindent
\textbf{Acknowledgment.} This work has been supported by the NSF grant
DMS 1853001. The author would like to thank Professor Yao Yao for
valuable discussions.

\bibliographystyle{plain}
\bibliography{singular}

\end{document}